\documentclass{amsart}
\usepackage[utf8]{inputenc}

\usepackage{amsmath, amssymb, amsthm}
\usepackage{enumerate, hyperref, tikz}
\usepackage[center]{caption}
\usetikzlibrary{topaths,calc}

\theoremstyle{plain}
\newtheorem{theorem}{Theorem}
\newtheorem{lemma}[theorem]{Lemma}
\newtheorem{proposition}[theorem]{Proposition}
\newtheorem{corollary}[theorem]{Corollary}

\theoremstyle{definition}
\newtheorem{definition}[theorem]{Definition}

\newcommand{\N}{\mathbb{N}}
\newcommand{\Pp}{\mathcal{P}}

\title{A note on increasing paths in countable hypergraphs}

\author{Valentino Vito}
\address{Faculty of Computer Science\\
Universitas Indonesia\\
Depok 16424, Indonesia}
\email{valentino.vito11@ui.ac.id}

\begin{document}


\begin{abstract}
An old result of M\"uller and R\"odl states that a countable graph $G$ has a subgraph whose vertices all have infinite degree if and only if for any vertex labeling of $G$ by positive integers, an infinite increasing path can be found. They asked whether an analogous equivalence holds for edge labelings, which Reiterman answered in the affirmative. Recently, Arman, Elliott, and R\"odl extended this problem to linear $k$-uniform hypergraphs $H$ and generalized the original equivalence for vertex labelings. They asked whether Reiterman's result for edge labelings can similarly be extended. We confirm this for the case where $H$ admits only finitely many Berge cycles.
\end{abstract}

\maketitle

\section{Introduction}

Let $G$ be a countable graph, and suppose that $\phi\colon V(G) \to \N$ is a bijective vertex labeling of $G$. An \emph{infinite increasing path} in $G$ is an infinite path $v_1v_2 \dots$ of vertices such that $\phi(v_i) < \phi(v_{i+1})$ for $i \ge 1$. Likewise, given a bijective edge labeling $\phi\colon E(G) \to \N$ of $G$, an infinite increasing path in $G$ is an infinite path $e_1e_2 \dots$ of edges such that $\phi(e_i) < \phi(e_{i+1})$ for $i \ge 1$. Throughout the rest of this note, vertex and edge labelings are always assumed to be bijective.

A necessary and sufficient condition for the existence of an infinite increasing path in $G$ given an arbitrary vertex (resp.\ edge) labeling has long been known. Theorem \ref{thm_graph} is the combined result of M\"uller and R\"odl \cite{muller}---who proved (i) $\leftrightarrow$ (ii) in 1982---and Reiterman \cite{reiterman} who showed that (i) $\leftrightarrow$ (iii) in 1989, answering the problem posed in \cite{muller}.

\begin{theorem}[M\"uller and R\"odl \cite{muller} and Reiterman \cite{reiterman}]\label{thm_graph}
Let $G = (V, E)$ be a countable graph. The following are equivalent:
\begin{enumerate}[\normalfont(i)]
    \item $G$ has a subgraph $G'$ in which every vertex has infinite degree in $G'$.
    \item For any vertex labeling $\phi\colon V \to \N$, there exists an infinite increasing path in $G$.
    \item For any edge labeling $\phi\colon E \to \N$, there exists an infinite increasing path in $G$.
\end{enumerate}
\end{theorem}

Recently, Arman, Elliott, and R\"odl \cite{arman} extended this problem to countable hypergraphs. Namely, they formulated the notion of an infinite increasing (loose) path on $k$-uniform hypergraphs, along with key results such as Theorems \ref{thm_hyper_vertex} and \ref{thm_hyper}. A comprehensive treatment on this topic and its variants can be found in Elliott's doctoral thesis \cite{elliott}.

\begin{definition}\label{ip}
Let $H$ be a $k$-uniform hypergraph. An \emph{infinite path} in $H$ is a sequence $v_1v_2\dots$ of distinct vertices of $H$, along with its corresponding sequence $e_0e_1\dots$ of edges such that $e_i = \{v_{(k-1)i+1}, v_{(k-1)i+2}, \dots, v_{(k-1)i+k}\}$ for $i \ge 0$.
\end{definition}

\begin{definition}\label{iip}
Let $H = (V, E)$ be a countable $k$-uniform hypergraph.
\begin{enumerate}
    \item Given a vertex labeling $\phi\colon V \to \N$, an \emph{infinite skip-increasing path} in $H$ is an infinite path $v_1v_2\dots$ with $\phi(v_{(k-1)i+1}) < \phi(v_{(k-1)i+k})$ for $i \ge 0$. It is an \emph{infinite increasing path} if $\phi(v_j) < \phi(v_{j+1})$ for $j \ge 1$.
    \item Given an edge labeling $\phi\colon E \to \N$, an \emph{infinite increasing path} in $H$ is an infinite path $e_0e_1\dots$ with $\phi(e_i) < \phi(e_{i+1})$ for $i \ge 0$.
\end{enumerate}
\end{definition}

The class of properties provided below generalizes condition (i) of Theorem \ref{thm_graph} to $k$-uniform hypergraphs.

\begin{definition}
A $k$-uniform hypergraph $H = (V, E)$ has property $\Pp_\ell$, $\ell \le k$, if there exists a nonempty set $V' \subseteq V$ such that for every $v \in V'$, we have
\[|\{e \in E: v \in e \text{ and } |V' \cap e| \ge \ell\}| = \infty.\]
\end{definition}

Note that $\Pp_\ell$ becomes more restrictive as $\ell$ gets larger. The cases $\ell = 2$ and $\ell = k$ are highly relevant to our problem, as the following theorems indicate.

\begin{theorem}[Arman, Elliott, and R\"odl \cite{arman}]\label{thm_hyper_vertex}
Let $H = (V, E)$ be a countable linear $k$-uniform hypergraph. The following are equivalent:
\begin{enumerate}[\normalfont(i)]
    \item $H$ has property $\Pp_k$.
    \item For any vertex labeling $\phi\colon V \to \N$, there exists an infinite increasing path in $H$.
\end{enumerate}
\end{theorem}

\begin{theorem}[Arman, Elliott, and R\"odl \cite{arman}]\label{thm_hyper}
Let $H = (V, E)$ be a countable linear $k$-uniform hypergraph. Consider the following statements:
\begin{enumerate}[\normalfont(i)]
    \item $H$ has property $\Pp_2$.
    \item For any vertex labeling $\phi\colon V \to \N$, there exists an infinite skip-increasing path in $H$.
    \item For any edge labeling $\phi\colon E \to \N$, there exists an infinite increasing path in $H$.
\end{enumerate}
We have \normalfont{(ii) $\leftrightarrow$ (i) $\rightarrow$ (iii)}.
\end{theorem}

Theorem \ref{thm_hyper} partially extends Theorem \ref{thm_graph} to the hypergraph setting. The critical missing piece to complete this extension is the implication (iii) $\rightarrow$ (i), which is left open in \cite{arman}.

\begin{definition}
A \emph{Berge cycle} in a hypergraph $H$ is a sequence $e_1v_1e_2v_2\dots e_nv_ne_1$, $n \ge 3$, of alternating edges and vertices of $H$ such that:
\begin{enumerate}[\normalfont(i)]
    \item $v_1, \dots, v_n$ are distinct vertices, and $e_1, \dots, e_n$ are distinct edges.
    \item For $1 \le i \le n$, the vertex $v_i$ belongs to $e_i$ and $e_{i+1}$ (we set $e_{n+1} = e_1$).
\end{enumerate}
\end{definition}

In this note, we show that the implication (iii) $\rightarrow$ (i) of Theorem \ref{thm_hyper} holds whenever $H$ contains only finitely many distinct Berge cycles. In fact, we prove a slightly stronger result: if for all $v \in V(H)$ there exist only finitely many distinct Berge cycles involving $v$, then (iii) $\rightarrow$ (i) of Theorem \ref{thm_hyper} holds. This marks significant progress in extending the combined work of M\"uller, Reiterman, and R\"odl. The main idea of the proof is to first reduce the edge labeling problem on $H$ to a vertex labeling problem on its dual $H^*$. The problem is further reduced to a graph labeling problem using certain graphs introduced in Definition \ref{skeleton}. This reduction is done so that Theorem \ref{thm_graph} can successfully be applied to our problem.

\section{Hypergraph labeling duality}

To motivate our approach in tackling the problem, we show that infinite increasing paths involving a vertex and edge labeling, respectively, are dual concepts in the hypergraph setting.

Given distinct vertices $v_1, \dots, v_n$ and edges $e_1, \dots, e_n$ in a hypergraph $H$, the sequence $v_1e_1v_2e_2\dots v_ne_n$ is a \emph{Berge path} if for $1 \le i \le n$, the vertex $v_i$ belongs to $e_{i-1}$ and $e_i$ (we set $e_0 = e_1$). We can define a Berge path in the form $v_1e_1\dots v_ne_nv_{n+1}$ in a similar manner so that $v_{n+1}$ belongs to $e_n$ in particular. The corresponding definition for infinite Berge paths is as follows.

\begin{definition}\label{beta-path}
An \emph{infinite Berge path} in a hypergraph $H$ is a sequence $v_1e_1v_2e_2\dots$ of alternating vertices and edges of $H$ such that:
\begin{enumerate}[\normalfont(i)]
    \item $v_1, v_2, \dots$ are distinct vertices, and $e_1, e_2, \dots$ are distinct edges.
    \item For $i \ge 1$, the vertex $v_i$ belongs to $e_{i-1}$ and $e_i$ (we set $e_0 = e_1$).
\end{enumerate}
\end{definition}

Unlike infinite paths for $k$-uniform hypergraphs $H$ given in Definition \ref{ip}, infinite Berge paths are defined for any hypergraph $H$. On $k$-uniform hypergraphs, $k \ge 3$, the notion of infinite paths is strictly stronger than infinite Berge paths. When $k = 2$, the two notions coincide.

\begin{definition}
Let $H = (V, E)$ be a countable hypergraph.
\begin{enumerate}
    \item Given a vertex labeling $\phi\colon V \to \N$, an \emph{infinite increasing Berge path} in $H$ is an infinite Berge path $v_1e_1v_2e_2\dots$ with $\phi(v_i) < \phi(v_{i+1})$ for $i \ge 1$.
    \item Given an edge labeling $\phi\colon E \to \N$, an \emph{infinite increasing Berge path} in $H$ is an infinite Berge path $v_1e_1v_2e_2\dots$ with $\phi(e_i) < \phi(e_{i+1})$ for $i \ge 1$.
\end{enumerate}
\end{definition}

Let us recall that the \emph{dual} $H^* = (V^*, E^*)$ of a hypergraph $H = (V, E)$ is defined by
\[V^* = E \quad \text{and} \quad E^* = \{v^*: v \in V\},\]
where $v^* = \{e \in E: v \in e\}$. We note that $(H^*)^* = H$.

\begin{proposition}\label{dual}
Let $H$ be a countable hypergraph. The following are equivalent:
\begin{enumerate}[\normalfont(i)]
    \item For any vertex labeling $\phi\colon V(H) \to \N$, there exists an infinite increasing Berge path in $H$.
    \item For any edge labeling $\phi\colon E(H^*) \to \N$, there exists an infinite increasing Berge path in $H^*$.
\end{enumerate}
\end{proposition}

\begin{proof}
The map $v_1e_1v_2e_2v_3\ldots \mapsto e_1v_2^*e_2v_3^*\dots$ transforms an infinite increasing Berge path given a vertex labeling $v \mapsto \phi(v)$ (resp.\ edge labeling $e \mapsto \phi(e)$) in $H$ to an infinite increasing Berge path given the edge labeling $v^* \mapsto \phi(v)$ (resp.\ vertex labeling $e \mapsto \phi(e)$) in $H^*$. A proof of the proposition easily follows.
\end{proof}

We introduce the dual property ${\Pp_2}^*$ of $\Pp_2$ with the following definition. Observe that $H$ has property $\Pp_2$ if and only if $H^*$ has property ${\Pp_2}^*$.

\begin{definition}
A hypergraph $H = (V, E)$ has property ${\Pp_2}^*$ if there exists a nonempty set $E' \subseteq E$ such that for every $e \in E'$, we have
\[|\{v \in e: v \in f \text{ for some edge } f \neq e \text{ in } E'\}| = \infty.\]
\end{definition}

We can now obtain an augmented version of Theorem \ref{thm_graph}. In the following corollary, (i) $\leftrightarrow$ (ii), (iii) $\leftrightarrow$ (iv), and (v) $\leftrightarrow$ (vi) are obtained through duality, whereas (i) $\leftrightarrow$ (iii) $\leftrightarrow$ (v) is due to Theorem \ref{thm_graph}.

\begin{corollary}\label{thm_1}
Let $G$ be a countable graph. The following are equivalent:
\begin{enumerate}[\normalfont(i)]
    \item $G$ has a subgraph $G'$ in which every vertex has infinite degree in $G'$.
    \item $G^*$ has property ${\Pp_2}^*$.
    \item For any vertex labeling $\phi\colon V(G) \to \N$, there exists an infinite increasing path in $G$.
    \item For any edge labeling $\phi\colon E(G^*) \to \N$, there exists an infinite increasing Berge path in $G^*$.
    \item For any edge labeling $\phi\colon E(G) \to \N$, there exists an infinite increasing path in $G$.
    \item For any vertex labeling $\phi\colon V(G^*) \to \N$, there exists an infinite increasing Berge path in $G^*$.
\end{enumerate}
\end{corollary}

Observe that Corollary \ref{thm_1} effectively extends the original equivalence (i) $\leftrightarrow$ (iii) $\leftrightarrow$ (v) for graphs $G$ to linear 2-regular hypergraphs $G^*$ with the equivalence (ii) $\leftrightarrow$ (iv) $\leftrightarrow$ (vi).

\section{Main result}

Theorem \ref{thm_2} contains the main contribution of this note. First, the following definition introduces a family of graphs that can be obtained from a hypergraph $H$. Figure \ref{fig} illustrates the concept.

\begin{definition}\label{skeleton}
Let $H = (V, E)$ be a hypergraph and $S = \{(v, e) \in V \times E: v \in e\}$. Suppose that $T$ is a subset of $S$ such that for every $f \in E$, the set $T_f = \{(v, e) \in T: e = f\}$ is finite. The \emph{skeleton of $H$ generated by $T$} is a graph $G$ defined by
\[V(G) = V \quad \text{and} \quad E(G) = \bigcup_{(v,e) \in T} \{\{v, w\}: w \in e \text{ and } w \neq v\}.\]
\end{definition}

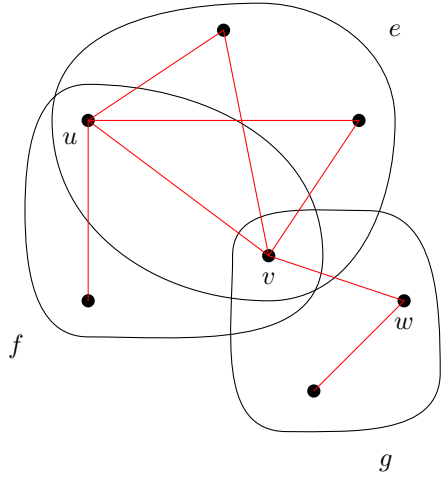
\begin{figure}
\centering
\begin{tikzpicture}[x=12mm,y=12mm]
    \node (v1) at (0,2) {};
    \node (v2) at (1.5,3) {};
    \node (v3) at (3,2) {};
    \node (v4) at (0,0) {};
    \node (v5) at (2,0.5) {};
    \node (v6) at (3.5,0) {};
    \node (v7) at (2.5,-1) {};

    \begin{scope}[fill opacity=0.8]
    \filldraw[fill=none] ($(v1)+(-0.4,0)$) 
        to[out=90,in=180] ($(v2) + (0.4,0.3)$)
        to[out=0,in=90] ($(v3) + (0.4,0)$)
        to[out=270,in=0] ($(v5) + (0,-0.5)$)
        to[out=180,in=270] ($(v1)+(-0.4,0)$);
    \filldraw[fill=none] ($(v4)+(-0.7,1)$)
        to[out=90,in=180] ($(v1)+(0,0.4)$)
        to[out=0,in=90] ($(v5)+(0.6,0)$)
        to[out=270,in=0] ($(v4)+(0,-0.4)$)
        to[out=180,in=270] ($(v4)+(-0.7,1)$);
    \filldraw[fill=none] ($(v5)+(-0.4,0)$)
        to[out=90,in=180] ($(v6)+(-0.4,1)$)
        to[out=0,in=90] ($(v6)+(0.4,-0.8)$)
        to[out=270,in=0] ($(v7)+(-0.3,-0.45)$)
        to[out=180,in=270] ($(v5)+(-0.4,0)$);
    \end{scope}
    
    \foreach \v in {1,2,...,7} {
        \fill (v\v) circle (2.5pt);
    }
    
    \draw[color=red] (2.5,-1) -- (3.5,0) -- (2,0.5) -- (1.5,3) -- (0,2);
    \draw[color=red] (0,0) -- (0,2) -- (2,0.5) -- (3,2) -- (0,2);
    
    \node at (3.4,3) {$e$};
    \node at (3.3,-1.8) {$g$};
    \node at (-0.8,-0.5) {$f$};
    \node at (-0.2,1.8) {$u$};
    \node at (2,0.25) {$v$};
    \node at (3.5,-0.25) {$w$};
\end{tikzpicture}

\caption{A hypergraph pictured along with a skeleton generated by $T = \{(u, e), (v, e), (u, f), (w, g)\}$; the skeleton is indicated by the red edges.}
\label{fig}
\end{figure}

\begin{theorem}\label{thm_2}
Let $H$ be a countable linear $k$-uniform hypergraph such that for all $v \in V(H)$, there exist only finitely many distinct Berge cycles involving $v$. The following statements are equivalent:
\begin{enumerate}[\normalfont(i)]
    \item $H$ has property $\Pp_2$.
    \item For any edge labeling $\phi\colon E(H) \to \N$, there exists an infinite increasing path in $H$.
    \item For any vertex labeling $\phi\colon V(H^*) \to \N$, there exists an infinite increasing Berge path in $H^*$.
    \item There exists a skeleton $G$ of $H^*$ such that for any vertex labeling $\phi\colon V(G) \to \N$, there exists an infinite increasing path in $G$.
    \item There exists a skeleton $G$ of $H^*$ that has a subgraph $G'$ in which every vertex has infinite degree.
\end{enumerate}
\end{theorem}

From the assumptions on $H$ in Theorem \ref{thm_2}, we gather that its dual $H^*$ is a countable linear $k$-regular hypergraph such that for all $e \in E(H^*)$, there exist only finitely many distinct Berge cycles involving $e$.

\begin{lemma}\label{lem_paths}
Let $H = (V, E)$ be a linear hypergraph with $f \in E$ and $u, v, w \in V$, where $v \neq w$. If $P_1 = u \dots vf$ and $P_2 = u \dots wf$ are Berge paths in $H$, then there exists a Berge cycle in the form $fv \dots wf$ in $H$.
\end{lemma}

\begin{proof}
Let $fv \dots u \dots wf$ be an alternating sequence of edges and vertices obtained by traversing through $P_1$ backward followed by $P_2$ forward. Remove the final $f$, and label the resulting sequence as follows:
\[fv \dots u \dots w = e_1u_1 \dots e_nu_n.\]
We will reduce the sequence $e_1u_1 \dots e_nu_n$ to a Berge path in the form $f_1v_1\dots f_mv_m$. Set $f_1 = f$ and $v_1 = v$. For $i \ge 2$, recursively define
\[f_i = e_{\max\{j: v_{i-1} \in e_j\}} \quad \text{and} \quad v_i = u_{\max\{j: u_j \in f_i\}}\]
until we obtain $v_m = w$ for some $m$, which is guaranteed to occur eventually. After suffixing the previously removed final $f$, we can show that $f_1v_1\dots f_mv_mf = fv \dots wf$ is the desired Berge cycle by showing that $m \ge 3$. Indeed, we cannot have a Berge cycle in the form $fvewf$ since $H$ is linear. This completes the proof.
\end{proof}

Now we are ready to prove our main result.

\begin{proof}[Proof of Theorem \ref{thm_2}]
(i) $\rightarrow$ (ii): This is precisely the implication (i) $\rightarrow$ (iii) of Theorem \ref{thm_hyper}.

(ii) $\rightarrow$ (iii): From (ii), it follows that there exists an infinite increasing Berge path in $H$ given any edge labeling $\phi\colon E(H) \to \N$. This condition is equivalent to (iii) by Proposition \ref{dual}.

(iii) $\rightarrow$ (iv): Let $S = \{(v, e) \in V(H^*) \times E(H^*): v \in e\}$. For each component $C$ of $H^*$, take an arbitrary vertex $v_C \in V(C)$. Now set
\[T = \{(v, e) \in S: \text{there exists a Berge path in the form } v_C \dots ve\}.\]
We claim that $T$ generates a skeleton. Assume to the contrary that there exists an $f \in E(H^*)$ such that $T_f = \{(v, e) \in T: e = f\}$ is an infinite set. If follows that there are infinitely many distinct vertices $v$ such that a Berge path in the form $v_C \dots vf$ exists, where $C$ is such that $f \in E(C)$. Applying Lemma \ref{lem_paths} to pairs of such Berge paths, we obtain infinitely many distinct Berge cycles in $H^*$ involving $f$, contradicting our assumption.

Let $G$ denote the skeleton generated by $T$, and suppose that $\phi\colon V(G) \to \N$ is a vertex labeling of $G$. We show that there exists an infinite increasing path in $G$ given $\phi$. By assumption, there exists an infinite increasing Berge path $P = v_1e_1v_2e_2\dots$ in $H^*$. Suppose that $P$ lies in some component $C$, and take a Berge path in the form $P_1 = v_C \dots v_1$. Now pick an integer $j$ such that the infinite subpath $v_je_jv_{j+1}e_{j+1} \dots$ of $P$ has no edge in common with $P_1$. Fix $i \ge j$, and let $P_2 = v_1e_1 \dots v_ie_i$ be a finite subpath of $P$. From $P_1$ and $P_2$ in succession, we obtain an alternating sequence $v_C \dots v_1e_1 \dots v_ie_i$ that can be reduced in a similar manner as in the proof of Lemma \ref{lem_paths} to a Berge path in the form $v_C \dots v_ie_i$. It follows that $(v_i,e_i) \in T$. Since $i$ was arbitrarily fixed, we have $\{v_i, v_{i+1}\} \in E(G)$ for every $i \ge j$. This gives an infinite increasing path $v_jv_{j+1}\dots$ in $G$.

(iv) $\rightarrow$ (v): This is evident from the implication (ii) $\rightarrow$ (i) of Theorem \ref{thm_graph}.

(v) $\rightarrow$ (i): We prove that $H^*$ has property ${\Pp_2}^*$. Let $T$ be the generator of $G$, and define
\[E' = \{e \in E(H^*): |V(G') \cap e| = \infty\}.\]
We first claim that for every $v \in V(G')$, there exists a corresponding edge $e \in E'$ such that $(v, e) \in T$. Fix $v \in V(G')$, and assume to the contrary that $(v, e) \notin T$ for every $e \in E'$. Now set
\[U = N_{G'}(v) \setminus \bigcup_{\substack{f \in E(H^*) \setminus E' \\ v \in f}} f.\]
We have $|U| = \infty$ since $v$ belongs to only finitely many edges in $H^*$ and each $f \notin E'$ contains only finitely many vertices of $G'$. By the definition of $U$ and the assumption that $(v,e) \notin T$ for $e \in E'$, we infer that whenever $(v,f) \in T$, we have $u \notin f$ for every $u \in U$. Also note that $\{u, v\} \in E(G)$ for all $u \in U$. It can thus be shown from Definition \ref{skeleton} that for $u \in U$, there exists an edge $e_u \in E(H^*)$ such that $(u, e_u) \in T$ and $v \in e_u$. Since each $T_{e_u} = \{(v, e) \in T: e = e_u\}$ is a finite set, we have $|\{e_u: u \in U\}| = \infty$. This implies that $\deg_{H^*}(v) = \infty$, which is a contradiction.

From the previous claim, it is clear that $E'$ is nonempty. Now fix $e \in E'$. Pick infinitely many distinct vertices $v \in V(G') \cap e$, and choose for each vertex a corresponding edge $f \in E'$ with $(v, f) \in T$. Since $T_e$ is finite, only finitely many of the chosen vertices have $e$ as their corresponding edge. The rest of the vertices have an edge other than $e$ as a corresponding edge, so
\[|\{v \in e: v \in f \text{ for some edge } f \neq e \text{ in } E'\}| = \infty\]
as desired.
\end{proof}

\section{Final remarks}

Concerning statements (i), (ii), and (iii) of Theorem \ref{thm_hyper}, it is considerably more difficult to verify implication (iii) $\rightarrow$ (i) than to prove (ii) $\rightarrow$ (i). While Theorem \ref{thm_2} only yields partial results, we conjecture that (iii) $\rightarrow$ (i) holds in its full generality. If this is the case, then the equivalence of (i), (ii), and (iii) follows.

\section*{Acknowledgement}

The author would like to thank Bradley Elliott for providing valuable directions regarding the problem at the beginning of this study.

\bibliographystyle{amsplain}
\bibliography{main}

\end{document}